\documentclass[runningheads]{llncs}

\usepackage{graphicx}
\usepackage{microtype}
\usepackage{amssymb, amsmath, mathtools}

\begin{document}

\title{Constrained Global Optimization by Smoothing\thanks{Supported by the Volkswagen Foundation (Az.: 9C090, 2022), the German Research Foundation (Project 416228727–SFB 1410), and the National Research Fund of Ukraine (Project 2020.02/0121).}}

\author{Vladimir Norkin\inst{1, 2} \and Alois Pichler\inst{3} \and Anton Kozyriev\inst{2}}
\authorrunning{Vladimir Norkin et al.}

\institute{V.M.Glushkov Institute of Cybernetics, National Academy of Sciences of Ukraine, Kyiv, 03178 Ukraine \and Igor Sikorsky Kyiv Polytechnic Institute, National Technical University of Ukraine, Kyiv, 03056, Ukraine \and Department of Mathematics, Chemnitz University of Technology, Chemnitz, 09126, Germany}

\maketitle

\begin{abstract}
This paper proposes a novel technique called "successive stochastic smoothing" that optimizes nonsmooth and discontinuous functions while considering various constraints. Our methodology enables local and global optimization, making it a powerful tool for many applications. First, a constrained problem is reduced to an unconstrained one by the exact nonsmooth penalty function method, which does not assume the existence of the objective function outside the feasible area and does not require the selection of the penalty coefficient. This reduction is exact in the case of minimization of a lower semicontinuous function under convex constraints. Then the resulting objective function is sequentially smoothed by the kernel method starting from relatively strong smoothing and with a gradually vanishing degree of smoothing. The finite difference stochastic gradient descent with trajectory averaging minimizes each smoothed function locally. Finite differences over stochastic directions sampled from the kernel estimate the stochastic gradients of the smoothed functions. We investigate the convergence rate of such stochastic finite-difference method on convex optimization problems. The "successive smoothing" algorithm uses the results of previous optimization runs to select the starting point for optimizing a consecutive, less smoothed function. Smoothing provides the "successive smoothing" method with some global properties. We illustrate the performance of the "successive stochastic smoothing" method on test-constrained optimization problems from the literature.

\end{abstract}

\section{Introduction}
Global optimization is a relevant area of research. It is the subject of a large number of articles in the fundamental encyclopedia of optimization \cite{Floudas2009}, as well as of recent monographs \cite{Horst_Pardalos_2013,Sergeyev_Kvasov_2017,Strongin_Sergeyev_2014} among others. Most global optimization methods are designed to optimize Lipschitz or smooth functions. The focus of the present paper is on optimization of non-Lipschitzian continuous as well as  discontinuous functions, where gradient and subgradient methods have no sense.

The need for global optimization of nonsmooth and discontinuous functions arises in many applications: optimal control in parametric strategies \cite[sec. 15.1.3]{Floudas_Pardalos_1999}, \cite{Batukhtin_Mayboroda_1995,Knopov_Norkin_2022}, training neural networks with discontinuous activation functions \cite{longo2022rham} and neural networks with categorical features \cite{Ustimenko_Prokhorenkova_2020}, engineering optimization \cite{Conn_Mongeau_1998,Imo_Leech_1984,Yang_2010}, probability optimization with discrete distribution, lexicographic optimization, two-level optimization. Note that difficult continuous problems with nonconvex inequality constraints can simply be reduced to unconditional discontinuous optimization problems using exact discontinuous penalty functions \cite{Batukhtin_1993}.

Relatively few ideas and methods are used for global optimization of discontinuous functions: branch and bound methods, group random search methods (genetic algorithms, simulated annealing), Nelder-Mead method, approximation and smoothing methods. Problems of discontinuous optimization with constraints are generally little discussed.

In a number of works, V.D. Batukhtin with co-authors \cite{Batukhtin_Mayboroda_1995}, \cite{Batukhtin_1993} proposed an approximate method for optimizing discontinuous functions, in which the original function is locally approximated by linear functions and is minimized by generalized gradient methods that use gradients of linear approximations. In these works, the theory of generalized approximation gradients of discontinuous functions is also constructed and the necessary optimality conditions are obtained.

In \cite{Gup79,MGN87} a smoothing method was developed and substantiated for local unconstrained optimization of non-smooth Lipschitz functions and in \cite{ENW95,Gupal_Norkin_1977} it was extended to discontinuous functions. In this method, the initial function is approximated by a sequence of smooth averaged functions, and iterative stochastic optimization methods are used for local optimization of the latter. In \cite{ENW95} the theory of generalized mollified gradients of discontinuous functions is constructed, connections with other known concepts of generalized gradients are investigated and necessary optimality conditions are obtained. The advantage of the proposed smoothing method is its geometric transparency and the possibility to  apply stochastic gradient methods to optimization of nonconvex non-smooth and discontinuous functions. 

But the disadvantage of both approaches, approximation and smoothing, is their local nature, i.e., focusing on finding local minima.

In this article, we further develop the smoothing method to solve global nonsmooth and discontinuous optimization problems under general (basically convex) constraints. Previously, we reduce the problem of conditional optimization to a problem of unconstrained optimization by exact penalty function methods. The smoothing method was invented by V. Steklov (W. Stekloff) \cite{Steklov_1907} and has long been used in analysis and computational mathematics \cite{Chagas_2017,ENW95}. In the works \cite{Gup79,Mayne_Polak_1984,MGN87} local Steklov smoothing, i.e. averaging over hypercubes was used to study and optimize non-smooth Lipschitz functions by stochastic finite-difference methods. To ensure the local convergence of the method to stationary points of the optimization problem, the smoothing parameter is directed to zero in coordination with steps of the method. Subsequently, the finite-difference methods of local optimization by the idea of smoothing were considered, for example, in the works \cite{Duchi_2015,Nesterov_Spokoiny_2017,shamir2017optimal}. In these methods usually the smoothing parameter is kept fixed and/or quite small, i.e. smoothing is used just for local smoothing.

The important difference between the above mentioned works and the present article is that the latter uses an iterative variable smoothing, starting with a fairly large smoothing parameter and ending with a small parameter. This strategy of changing the smoothing parameter gives the method some global  properties (see \cite[pp. 135-137]{MGN87}, \cite{Norkin_2020}). As the smoothing parameter decreases, the averaging functions uniformly or graphically approximate  the original function, so the minima of the approximate functions approach to the  minima of the original function. Regarding global optimization, the idea of the smoothing method is that smoothing eliminates shallow small local minima and does not much change the deep minima.

One more interesting application of the smoothing idea in optimization was proposed in \cite{Arikan_Burachik_Kaya_2020,Burachik_Kaya_2021}. The authors described the behavior of critical points of the Steklov smoothed functions when changing the smoothing parameter by a system of ordinary differential equations. If one takes as a starting point a global minimum of a strongly smoothed function, then by solving these equations one potentially can go over a critical curve to a global minimum of the original function. Remark that this idea was also discussed in \cite[pp. 135-137]{MGN87}, \cite{Norkin_2020}. However, this differentiable path following approach can only be applied to the unconditional global optimization of functions given by analytical expressions, for example, for polynomial functions.

In this paper, we propose a successive stochastic smoothing method combined with critical path following for optimization of continuous and discontinuous functions under constraints. The structure of the proposed approach is the following.

First, the constrained problem is reduced an unconstrained one by some  exact penalty function method. In particular, it is proposed to use the exact penalty method from \cite{Galvan_2021,Norkin_2020}. The method does not require the existence of the objective function outside the feasible domain and does not require selection of penalty coefficient. Discontinuous exact penalty functions can also be employed \cite{Batukhtin_1993}.

For optimization of the obtained penalized unconstrained problems we use a successive kernel smoothing method with gradual vanishing of the smoothing parameter. Here kernel can be any probability density function dependent on smoothing (width) parameter. The local convergence of the method is theoretically validated for the case of non-Lipschitz continuous functions and for the so-called strongly lower semicontinuous functions \cite{ENW95}.

For the practical implementation of the smoothing method we exploit a stochastic finite-difference representation of gradients of the smoothed functions as in \cite{Steklov_1907,Gup79,MGN87,Nesterov_Spokoiny_2017,shamir2017optimal}, namely we use central finite differences in random directions as in \cite{Nesterov_Spokoiny_2017,shamir2017optimal}. Having such finite difference stochastic gradients, one can use contemporary stochastic optimization methods \cite{Bottou_Curtis_Nocedal_2018,Duchi_2011,Duchi_2015,Kingma_Ba_2017,Nemirovski_2009} for the optimization of the smoothed functions. In particular we use  stochastic finite-difference gradient method with trajectory averaging \cite{Nemirovski_2009} and with batch estimation of stochastic gradients. 

In the literature (see, e.g., \cite{Nesterov_Spokoiny_2017,shamir2017optimal}) the described approach is applied to the local optimization of nonsmooth Lipschitzian or smooth convex functions.  In the present paper we extend this approach to optimization of non-Lipschtzian continuous and discontinuous functions under convex constraints. An important difference of our modification of the approach is that we minimize a sequence of the smooth approximating functions starting from rough approximations and moving to more and more accurate approximations. This strategy provides the method some global properties as explained in \cite{Norkin_2020}, the method avoids some shallow local minima. Another important element of our method is critical path following, when transition from the minimum point of one smoothed function to an initial point for the minimization of another smoothed function is carried out by the Gelfand-Zetlin-Nesterov valley step (cf. \cite{Gelfand_Tsetlin_1962,Nesterov_1983}). 

We also evaluate the performance of the proposed method on nonsmooth Lipschitzian convex functions. We show that in this case the rate of convergence is (roughly) proportional to $\sqrt{n/(Kt)}$, where $t$ is the iteration number, $n$ is the dimension of the problem space, $K$ is the batch size when estimating the stochastic gradients. This rate of convergence corresponds to results of \cite{Nesterov_Spokoiny_2017,shamir2017optimal}.

\section{Reduction of a constrained optimization problem to an unconstrained one by exact penalty function methods}
\label{sec:penalty}

Let it be necessary to solve a conditional global optimization problem:
\begin{equation}
    \label{eqn1}
    f(x)\to \underset{x\in D\subseteq \mathbb{R}^{n}}{\mathop{\min }},
\end{equation}
where $f(x)$ is a lower semicontinuous function defined on a closed set $ D \subseteq \mathbb{R}^{n} $ such that $ f(x) \to +\infty $ as $ \left \| x \right \| \to +\infty $ and $ x \in D $ ; ${\mathbb{R}}^{n}$ is $ n $-dimensional Euclidean space; $ \left \| \cdot \right \| $ is some norm in $ {{\mathbb{R}}^{n}} $.

To apply smoothing method for constrained optimization problems without having irregularities on boundaries, we first reduce the problem to an unconstrained one. There are several ways to do this. For example, if
$$ D=\left\{ x|{{g}_{j}}(x)\le 0,\,j=1,...,J\,;\,\,{{h}_{k}}(x)=0,\,k=1,...,K \right\} $$
then in the exact penalty function method, the Lipschitz function $ f(x) $ is replaced by functions:
$$ \Phi_1(x):=f\left( x \right)+M\left( \sum\nolimits_{j}{\max \left\{ 0,{{g}_{j}}(x) \right\}}+\sum\nolimits_{k}{\left| {{h}_{k}}(x) \right|} \right), $$
$$ \Phi_2(x):=f\left( x \right)+M\rho_D(x),\;\;\;\;\;\rho_D(x)={{\inf }_{y\in D}}\left\| y-x \right\|, $$
with a sufficiently large penalty multiplier $ M $ and then consider the problem of unconditional optimization $ \Phi_i(x),\;i=1,2, $ (cf. \cite[Proposition 9.68]{RocW98}). Here it is assumed that the functions $ f,{{g}_{j}},{{h}_{k}} $ are defined over the whole space $ {{\mathbb{R}}^{n}} $.

Let $ f : D \to {{\mathbb{R}}^{1}}$ and the set $D \subseteq{{\mathbb{R}}^{n}}$ in (\ref{eqn1}) is convex and closed, i.e., in the representation  $ D=\left\{ x|{{g}_{j}}(x)\le 0,\,j=1,...,J\,;\,\,{{h}_{k}}(x)=0,\,k=1,...,K \right\} $ functions  ${{g}_{j}}$ are continuous and convex, and $ {{h}_{k}} $ are linear. Denote $ {{\pi }_{D}}(x) $ the projection of the point $ x $ on the set $ D $. For a simple set $ D $ given by linear constraints, the problem of searching projection ${{\pi }_{D}}(x)$ is either solved analytically or reduced to a quadratic programming problem. We introduce penalty functions \cite{Galvan_2021,Norkin_2020,Norkin_2022}:
\begin{equation}
    \label{eqn2}
    F_1(x):=f\left( {{\pi }_{D}}(x) \right)+M\left( \sum\nolimits_{j=1}^{J}{\max \left\{ 0,{{g}_{j}}(x) \right\}}+\sum\nolimits_{k=1}^{K}{\left| {{h}_{k}}(x) \right|} \right),	 
\end{equation}
\begin{equation}
    \label{eqn3}
    F_2(x):=f\left( {\pi }_{D}(x)\right)+M\rho_D(x),\;\;\;\;M>0,
\end{equation}
and consider the unconstrained optimization problems:
\begin{equation}
    \label{eqn4}
    F_i(x)\to \min_{x\in\mathbb{R}^n},\;\;\; i=1,2...
\end{equation}
Note that in (\ref{eqn2}), (\ref{eqn3}) function $f$ may not be defined outside the feasible set $D$.
\begin{lemma}
    \label{Lemma 1} \cite{Galvan_2021}. Let function $f$ be lower semicontinuous on a non-empty closed convex set $D$, $M>0$. Then problems (\ref{eqn1}) and (\ref{eqn4}) are equivalent, i.e., each local (global) minimum of one problem is a local (global) minimum of the another, and the optimal values of the problems in the corresponding minima coincide.
\end{lemma}
\begin{remark}
    If $D$ is a convex closed set, then function ${{\rho }_{D}}(x)$ is continuous \cite[Example 1.20]{RocW98} and the mapping ${{\pi }_{D}}(x)$ is single valued and continuous on ${{\mathbb{R}}^{n}}$ \cite[Example 2.25]{RocW98}. If function $f$ is continuous (lower semicontinuous) on a convex closed set $D$, then function (\ref{eqn3}) is continuous (lower semicontinuous) on $\mathbb{R}^n$.
\end{remark}
\begin{lemma}
    \label{Lemma 2} \cite{Galvan_2021}. If function $f$ is Lipschitzian with constant $L$ on a convex set $D$, then the function $F_2(x)$ defined by equality (\ref{eqn3}) is also Lipschitzian with constant $(L+2M)$ on the whole space ${{\mathbb{R}}^{n}}$.
\end{lemma}
If some admissible point ${{x}_{0}}\in D$  is known, then the exact penalty function can be constructed as follows \cite{Norkin_2020,Norkin_2022}. Let $x\notin D$ and $y(x)$ be the nearest to $x$ point from the set $D$ lying on the segment connecting ${{x}_{0}}$ and $x$. Let us define the mapping
$${p}_{D}(x)=\left\{ \begin{matrix}
   x, & x\in D,  \\
   y(x), & x\notin D,  \\
\end{matrix} \right.$$
and the penalty functions ${r}_{D}(x)=\left\| x-{p}_{D}(x) \right\|$ and $F_3(x):=f({p}_{D}(x))+M{r}_{D}(x)$. Consider the unconstrained optimization problem:
\begin{equation}
    \label{eqn5}
    F_3(x):=f({p}_{D}(x))+M{r}_{D}(x)\to {\min }_{x\in \mathbb{R}^{n}}, \;\;\;M>0.
\end{equation}
\begin{lemma}
    \label{Lemma 3} \cite{Norkin_2022}. Let $D$ be a non-empty closed set. Then problems (\ref{eqn1}) and (\ref{eqn5}) are globally equivalent, i.e., the global minimum of one task is the global minimum of the other. Moreover, any local minimum of problem (\ref{eqn5}) is a local minimum of problem (\ref{eqn1}). If $D$ is a convex closed set and ${{x}_{0}}$ is an interior point of $D$, any local minimum of problem (\ref{eqn1}) is a local minimum of problem (\ref{eqn5}).
\end{lemma}

\section{Smoothed functions and their stochastic gradients}
\label{sec:smoothing}

The idea of smoothing is widely used in analysis and optimization in different forms. In what follows we use the so-called Steklov like smoothing (\cite{Steklov_1907}). This kind of smoothing allows to represent gradients of smoothed functions in a convenient for optimization finite-difference form.

Let $\mu (x)$ be some probability kernel on ${{\mathbb{R}}^{n}}$, i.e., $\mu (x)\ge 0$ and  $\int_{{{\mathbb{R}}^{n}}}{\mu (x)dx}=1$. Along with function $F(x)$, $x\in\mathbb{R}^n$, consider the so-called Steklov-Schwartz smoothed functions \cite{ENW95}
\begin{equation}
    {{F}_{h}}(x)=\frac{1}{{{h}^{n}}}\int_{{{\mathbb{R}}^{n}}}{F(x+y)\mu ({y}/{h}\;)dy}
    =\int_{{{\mathbb{R}}^{n}}}{F(x+hz)\mu (z)dz}.
        \label{Steklov-Schwartz}
\end{equation}
(it is assumed that the integrals exist), where $h>0$ is the smoothing parameter. If the function $F(x)$ is locally Lipschitz, then ${{F}_{h}}(x)$ is continuously differentiable and its gradient can be represented as
\begin{equation}
    \label{eqn6}
    \nabla {{F}_{h}}(x)=\frac{1}{{{h}^{n}}}\int_{{{\mathbb{R}}^{n}}}{\partial F(x+y)\mu ({y}/{h}\;)dy}=\int_{{{\mathbb{R}}^{n}}}{\partial F(x+hz)\mu (z)dz}
\end{equation}
(assuming that the integrals exist), where $\partial F(y)$ is the Clarke subdifferential \cite{Clarke_1990} of function $F(\cdot )$ at point $y$. The integral of a multivalued mapping here is understood as the set of integrals of measurable single-valued sections of the mapping $\partial F(\cdot )$ (\cite{Clarke_1990,MGN87}). For non-Lipschitzian functions the representation (\ref{eqn6}) is not defined, so the gradient sampling methods \cite{burke2020gradient} cannot be employed for minimization of $F$.

Classical Steklov smoothing \cite{Steklov_1907} uses uniform distribution on the cube $\{x\in\mathbb{R}^n:\max_i |x_i|\le h/2\}$. Properties of Steklov-smoothed functions were studied, e.g., in \cite{Chagas_2017,Gup79,MGN87}. In \cite{Gup79,MGN87} on the basis of Steklov smoothing, stochastic finite-difference optimization methods were developed for optimization of nonsmooth Lipschitz functions.

\subsection{Ball smoothing of continuous functions}
Consider smoothed (or averaged) functions of the following form (\cite{Duchi_2015,shamir2017optimal}):
\begin{equation}
    \label{ballsmoothing}
    {{F}_{h}}(x)=(v_1)^{-1}\int_{{{V}_{1}}}{F(x+hy)dy},
\end{equation}
where $V_1=\left\{ y\in {{\mathbb{R}}^{n}}:\left\| y \right\|_2\le 1 \right\}$ is the unit ball in $\mathbb{R}^n$ with volume $v_1$. The gradient $\nabla {{F}_{h}}(x)$, when $F(x)$ continuous, is calculated by the surface integral
\begin{eqnarray}
    \nabla {{F}_{h}}(x) &=& \frac{n}{s_1}\int_{{S}_{1}}\frac{1}{2h}{\left( F(x+hy)-F(x-hy) \right)yds,} \nonumber
\end{eqnarray}
where $S_1=\left\{ y\in {{\mathbb{R}}^{n}}:\left\| y \right\|_2=1 \right\}$ is the unit sphere in $\mathbb{R}^n$ with square $s_1$.

Denote $\tilde{y}$ a random vector uniformly distributed on the unit sphere ${{S}_{1}}$, then the gradient $\nabla {{F}_{h}}(x)$ can be represented as mathematical expectations (\cite{Duchi_2015,shamir2017optimal}):
\begin{eqnarray}
    \label{eqn9}
    \nabla {{F}_{h}}(x)&=&\frac{n}{h}{{\mathbb{E}}_{{\tilde{y}}}}\left( F(x+h\tilde{y})-F(x) \right)\tilde{y}
    =\frac{n}{2h}{{\mathbb{E}}_{{\tilde{y}}}}\left( F(x+h\tilde{y})-F(x-h\tilde{y}) \right)\tilde{y}.
\end{eqnarray}
For $L$-Lipschitz function $F$ it holds \cite{shamir2017optimal}:
\begin{equation}
    \label{shamirinequality}
    {{\mathbb{E}}_{{\tilde{y}}}}{{\left\| \frac{1}{2h}\left( F(x+h\tilde{y})-F(x-h\tilde{y}) \right)\tilde{y} \right\|}^{2}}\le C\cdot {{L}^{2}}/n,
\end{equation}
where $C$ is some absolute constant. Our experiments indicate that $C\le 1$.

\subsection{Smoothing of discontinuous functions}

Smoothing of discontinuous functions for optimization purposes has some peculiarities; smoothed functions do not uniformly approximate the original function and may be not differentiable. We limit the consideration to the case of the so-called strongly lower semicontinuous functions $F(x)$, which grow at the infinity not faster than some power of $\|x\|$.
\begin{definition}
    \label{Definition 1} (strongly lower semicontinuous functions, \cite{ENW95}). Function $F:\ {{\mathbb{R}}^{n}}\to {{\mathbb{R}}^{1}}$ is called lower semicontinuous (lsc) at point $x$ if $\underset{k\to \infty }{\mathop{\lim \inf }}\,F({{x}^{k}})\ge F(x)$ for all sequences ${{x}^{k}}\to x$.  Function $F:\ {{\mathbb{R}}^{n}}\to {{\mathbb{R}}^{1}}$ is called strongly lower semicontinuous (strongly lsc) at point $x$, if it is lower semicontinuous at $x$ and there exists a sequence ${{x}^{k}}\to x$ such that it is continuous at $x^{k}$ (for all $x^{k}$) and $F({{x}^{k}})\to F(x)$. Function $F$ is called strongly lower semicontinuous on $X\subseteq \mathbb{R}^{n}$ if this is the case for all $x\in X$.
\end{definition}
The property of strong lower semicontinuity is preserved under continuous transformations. Assume that function $F$ satisfy the condition $|F(x)|\le c_1+c_2\|x\|^{c_3}$ for some constants  $c_1, c_2, c_3\ge 0$ and let $\mu(y)={{(2\pi )}^{-n/2}}{{e}^{-\,{\|y\,\,\|{{\,}^{2}}}/{2}\;}}$, $y\in \mathbb{R}^n$, be the Gaussian kernel. Consider smoothed functions
\begin{eqnarray}
    \label{gaussmoothing}
    {F}_{h}(x)&=&\int_{\mathbb{R}^n}F(x+hy)\mu(y)dy
    =\frac{1}{(2\pi)^{n/2}h^n}\int_{\mathbb{R}^n}F(y)e^{-||y-x||^2/(2h^2)}dy.
\end{eqnarray}
If the function $F$ is not continuous, then we cannot expect the averaged functions ${{F}_{h}}(x)$ to converge to $F$ uniformly. But we don't need that. We need such a convergence of the averaged functions ${{F}_{h}}(x)$ to $F$ that guarantees the convergence of the minima of ${{F}_{h}}(x)$  to the minima of $F$. This property is provided by the so-called epi-convergence of functions \cite[Theorem 7.33]{RocW98}.
\begin{definition}		
    \label{Definition 2} (Epi-convergence \cite{RocW98}). A sequence of functions $\{{{F}^{k}}:{{\mathbb{R}}^{n}}\to \overline{\mathbb{R}},\ k\in \mathbb{N}\}$ epi-converge to a function $F:{{\mathbb{R}}^{n}}\to \overline{\mathbb{R}}$ at a point $x$, if

$\underset{k\to \infty }{\mathop{\lim \inf }}\,{{F}^{k}}({{x}^{k}})\ge F(x)$ for all ${{x}^{k}}\to x$;

$\underset{k\to \infty }{\mathop{\lim }}\,{{F}^{k}}({{x}^{k}})=F(x)$ for some sequence ${{x}^{k}}\to x$.

\noindent
    The sequence ${{\{{{F}^{k}}\}}_{k\in \mathbb{N}}}$ epi-converges to $F$, if this is the case at every $x\in {{\mathbb{R}}^{n}}$.
\end{definition}
\begin{theorem}
    \label{Theorem1} \cite{ENW95}. For a strongly semicontinuous locally integrable function $F:{{\mathbb{R}}^{n}}\to \mathbb{R}$, any associated sequence of averaged functions $\{{{F}_{h}},\ h\in {{\mathbb{R}}_{+}}\}$ epi-converges to $F$, that is, for any sequence ${{h}^{k}}\downarrow 0$ the averaged functions ${F}_{h_k}$ epi-converge to $F$.
\end{theorem}	
To optimize discontinuous functions, we approximate them by averaged functions. The convolution of a discontinuous function with the corresponding kernel (probability density) improves analytical properties of the resulting function, but increases the computational complexity of the problem, since it transforms the deterministic function into an expectation function, which is a multidimensional integral. Therefore, such an approximation makes sense only in combination with the corresponding stochastic optimization methods. First, we study conditions of continuity and continuous differentiability of the averaged functions.

Thus for the strongly lsc function $F$ the average functions ${{F}_{h}}$ epi-converge to $F$ as $h \downarrow 0$ and each function ${{F}_{h}}$ is analytical with the gradient \cite{Nesterov_Spokoiny_2017}:
\begin{eqnarray}
    \nabla {{F}_{h}}(x) &=& \frac{1}{h^{n+2}}\int_{\mathbb{R}^n}F(y)(y-x)
		\mu \left(\frac{y-x}{h} \right)dy \nonumber \\
    &=& \frac{1}{h}\int_{{{\mathbb{R}}^{n}}}F(x+hz)z\mu(z)dz
    = -\frac{1}{h}\int_{{{\mathbb{R}}^{n}}}F(x-hz)z\mu(z)dz\nonumber \\
    &=& {\mathbb{E}_{\eta }}\frac{1}{2h}[F(x+h\mathbf{\eta })-F(x-h\mathbf{\eta })]\mathbf{\eta },
    \label{gaussfindif}
\end{eqnarray}	
where the random variable ${\bf \eta}$ has the standard normal distribution $\mu(\cdot)$ in $\mathbb{R}^n$ and $\mathbb{E}_{\eta}$ denotes the mathematical expectation over ${\bf \eta}$. Thus, the random vector
\begin{equation}
    \label{eqn12}	
    {{\xi }_{h}}(x,\mathbf{\eta })=\frac{1}{2h}[F(x+h\mathbf{\eta })-F(x-h\mathbf{\eta })]\mathbf{\eta }
\end{equation}
is an estimate of the gradient $\nabla {{F}_{h}}(x)$. For $L$-Lipschitz function $F$ it holds \cite{Nesterov_Spokoiny_2017}:
\begin{equation}
    \label{Nesterov_Spokoiny}
    {{\mathbb{E}}_{{\tilde{y}}}}{{\left\| \frac{1}{2h}\left( F(x+h\tilde{y})-F(x-h\tilde{y}) \right)\tilde{y} \right\|}^{2}}\le n{{L}^{2}}
\end{equation}

\section{Stochastic finite-difference optimization methods for nonsmooth problems}

First, we investigate the rate of convergence of the stochastic finite-difference methods (with the trajectory averaging) on convex Lipschitz functions. The following theorem establishes the rate of convergence of the smoothing method in case of the batch calculation of gradient estimates. This method allows parallelization of the gradient estimation.
\begin{theorem}
    \label{Theorem 2} (A ball smoothing and a batch calculation of the finite-difference gradient estimates). Let's assume that the function $F(\cdot )$ is convex and Lipschitz with some constant $L$ (in the Euclidean norm) on a convex compact set $X\subseteq {{\text{R}}^{n}}$ and the estimate $\left\| X \right\|={{\sup }_{x\in X}}\left\| x \right\|\le D$ is true; the sequence $\left\{ {{x}_{t}} \right\}$ is built recursively according to the following rule:
    \begin{equation}
        \label{eqn13}
        {{x}_{t+1}}={{\pi }_{X}}\left( {{x}_{t}}-{{\rho }_{t}}{{\eta }_{t}} \right),\;\;\;{{x}_{1}}\in X, \;\;\; t=1,2, \ldots,
    \end{equation}
    random directions $\eta_{t}$  in  (\ref{eqn13}) are calculated as follows:
    \begin{equation}
        \label{eqn15} {{\eta }_{t}}=\frac{1}{K}\sum\limits_{k=1}^{K}{\frac{1}{2{{h}_{t}}}\left( F({{x}_{t}}+{{h}_{t}}\tilde{y}_{t}^{k})-F({{x}_{t}}-{{h}_{t}}\tilde{y}_{t}^{k}) \right)\tilde{y}_{t}^{k}},
    \end{equation}
    where $\left\{ \tilde{y}_{t}^{k} \right\}_{k=1}^{K}$ are independent random vectors uniformly distributed on the unit sphere with the center at the origin of coordinates. Then at a constant step $\rho =\frac{D\sqrt{nK}}{L\sqrt{2T(C+K/n)}}$ and a constant smoothing parameter $h_t=h$ the average point ${{\bar{x}}_{T}}=\frac{1}{T}\sum_{t=1}^{T}{{{x}_{t}}}$ satisfies the relation:
    \begin{equation}
        \mathbb{E}{{F}_{h}}({{\bar{x}}_{T}})-\min_{x\in X}{{F}_{h}}({{x}}) \le \frac{LD}{\sqrt{T}}\sqrt{\frac{2n}{K}}\left(C+\frac{K}{n}\right)^{1/2}. \nonumber
    \end{equation}
    For the variable steps $\rho_t =\frac{D\sqrt{nK}}{L\sqrt{2t(C+K/n)}}$ and $\bar{x}_t=\sum_{\tau=1}^t \rho_t x_t/\sum_{\tau=1}^t \rho_t$ it holds
    \begin{equation}
        \mathbb{E}{{F}_{h}}({{\bar{x}}_{t}})-\min_{x\in X}{{F}_{h}}({{x}}) \le \frac{DL}{\sqrt{2}}\sqrt{\frac{n}{K}}\left(C+\frac{K}{n}\right)^{1/2} \frac{2+\ln{t}}{\sqrt{t}-1}. \nonumber
    \end{equation}
    For the variable steps $\rho_t =\frac{D\sqrt{nK}}{L\sqrt{2t(C+K/n)}}$ and smoothing parameters $h_t=L\rho_t/K$ for the averaged points $\bar{x}_t=\sum_{\tau=1}^t \rho_t x_t/\sum_{\tau=1}^t \rho_t$ it holds
    \begin{eqnarray}
        \mathbb{E}{F}({\bar{x}_t})-\min_{x\in X}{{F}}({{x}}) &\le& \frac{DL}{2}\sqrt{\frac{n}{K}} \left(2+C+\frac{K}{n}\right)^{1/2}\frac{2+\ln{t}}{\sqrt{t}-1}.
    \end{eqnarray}
    Here $C$ is some absolute constant from inequality (\ref{shamirinequality}).
\end{theorem}
\begin{proof}
    Let $x^*$ be any fixed point in $X$. By convexity of $X$ the following inequalities are fulfilled:
    \[
        \begin{array}{lcl}
            {{\left\| {{x}^{*}}-{{x}_{t+1}} \right\|}^{2}}&=&{{\left\| {{x}^{*}}-{{\pi }_{X}}({{x}_{t}}-{{\rho }_{t}}{{\eta }_{t}}) \right\|}^{2}}\le {{\left\| {{x}^{*}}-{{x}_{t}}+{{\rho }_{t}}{{\eta }_{t}}) \right\|}^{2}} \\
            &=&{{\left\| {{x}^{*}}-{{x}_{t}} \right\|}^{2}}+2{{\rho }_{t}}\left\langle {{\eta }_{t}},{{x}^{*}}-{{x}_{t}} \right\rangle +\rho _{t}^{2}{{\left\| {{\eta }_{t}} \right\|}^{2}}.
        \end{array}
    \]
    Let us take the conditional mathematical expectation from both parts of this inequality with respect to $\sigma $-algebra generated by the random variable $\left\{ {{x}_{t}} \right\}$:
    $$\mathbb{E}\left[ \left. {{\left\| {{x}^{*}}-{{x}_{t+1}} \right\|}^{2}} \right|{{x}_{t}} \right]\le {{\left\| {{x}^{*}}-{{x}_{t}} \right\|}^{2}}+\frac{2{{\rho }_{t}}}{n}\left\langle \nabla {{F}_{h_t}}({{x}_{t}}),{{x}^{*}}-{{x}_{t}} \right\rangle +\rho _{t}^{2}\mathbb{E}\left[ \left. {{\left\| {{\eta }_{t}} \right\|}^{2}} \right|{{x}_{t}} \right]$$
    Due to convexity of the functions $F$ and ${{F}_{h_t}}$ we have: $\left\langle \nabla {{F}_{h_t}}({{x}_{t}}),{{x}^{*}}-{{x}_{t}} \right\rangle \le {{F}_{h_t}}({{x}^{*}})-{{F}_{h_t}}({{x}_{t}})$. As a result we obtain the inequality
    \begin{eqnarray}
        2\rho_t\left({{F}_{h_t}}({{x}_t})-{{F}_{h_t}}({{x}^*})\right)
        &\le& -n\mathbb{E}\left[ \left. {{\left\| {{x}^{*}}-{{x}_{t+1}} \right\|}^{2}} \right|{{x}_{t}} \right] \\
        &+& n{{\left\| {{x}^{*}}-{{x}_{t}} \right\|}^{2}} +\rho _{t}^{2}n\mathbb{E}\left[ \left. {{\left\| {{\eta }_{t}} \right\|}^{2}} \right|{{x}_{t}} \right] \nonumber.
    \end{eqnarray}
    Let us take the mathematical expectation of both sides of this inequality:
    \begin{eqnarray}
        2\rho_t\left(\mathbb{E}	{{F}_{h_t}}({{x}_t})-{{F}_{h_t}}({{x}^*})\right)
        &\le& -n\mathbb{E}{{\left\| {{x}^{*}}-{{x}_{t+1}} \right\|}^{2}} \\
        &+& n\mathbb{E}{{\left\| {{x}^{*}}-{{x}_{t}} \right\|}^{2}} + \rho _{t}^{2}n\mathbb{E}\text{E}\left[ \left. {{\left\| {{\eta }_{t}} \right\|}^{2}} \right|{{x}_{t}} \right] \nonumber.
    \end{eqnarray}
    Denote
    $\eta _{t}^{k}=\frac{1}{2{{h}_{t}}}\left( F({{x}_{t}}+{{h}_{t}}\tilde{y}_{t}^{k})-F({{x}_{t}}-{{h}_{t}}\tilde{y}_{t}^{k}) \right)\tilde{y}_{t}^{k};$
    $${{\bar{\eta}}_{t}}=\mathbb{E}\left[\eta_{t}^{k} |x_t\right]=\mathbb{E}\left[\frac{1}{2{{h}_{t}}}\left( F({{x}_{t}}+{{h}_{t}}\tilde{y}_{t}^{1})-F({{x}_{t}}-{{h}_{t}}\tilde{y}_{t}^{1}) \right)\tilde{y}_{t}^{1}|x_t\right].$$  
    By (\ref{eqn9}) it holds, $\mathbb{E}\left[ \eta _{t}^{k}|{{x}_{t}} \right]={{{\bar{\eta }}}_{t}} =n^{-1}\nabla F_{h_t}(x_t)$ for all $k$. For the expression $\mathbb{E}\left[ \left. {{\left\| {{\eta }_{t}} \right\|}^{2}} \right|{{x}_{t}} \right]$ the following estimates holds true:
    \[
        \begin{array}{lcl}
            \mathbb{E}\left[ \left. {{\left\| {{\eta }_{t}} \right\|}^{2}} \right|{{x}_{t}} \right]
            &=&\mathbb{E}\left[ {{\left\| \frac{1}{K}\sum\limits_{k=1}^{K}{\eta _{t}^{k}} \right\|}^{2}}|{{x}_{t}} \right]=\mathbb{E}\left[ {{\left\| \frac{1}{K}\sum\limits_{k=1}^{K}{\left( \eta _{t}^{k}-{{{\bar{\eta }}}_{t}} \right)}+{{{\bar{\eta }}}_{t}} \right\|}^{2}}|{{x}_{t}} \right] \\
            &\le& 
            2\mathbb{E}\left[ {{\left\| \frac{1}{K}\sum\limits_{k=1}^{K}{\left( \eta _{t}^{k}-{{{\bar{\eta }}}_{t}} \right)} \right\|}^{2}}|{{x}_{t}} \right]+2\mathbb{E}\left[ \left\| \bar{\eta}_t \right\|^2|x_t \right]\\
            &=&
            \frac{2}{{{K}^{2}}}\sum\limits_{k=1}^{K}{\mathbb{E}\left[ {{\left( \eta _{t}^{k}-{{{\bar{\eta }}}_{t}} \right)}^{2}}|{{x}_{t}} \right]}+2\mathbb{E}\left[ \left\| \bar{\eta}_t \right\|^2|x_t \right]\\
            &\le& 
            \frac{2}{{{K}^{2}}}\sum\limits_{k=1}^{K}{\left( \mathbb{E}\left[ {{\left\| \eta _{t}^{k} \right\|}^{2}}|{{x}_{t}} \right]-\mathbb{E}\left[ \left\| \bar{\eta}_t \right\|^2|x_t \right] \right)}+2\mathbb{E}\left[ \left\| \bar{\eta}_t \right\|^2|x_t \right]\\
            &=&
            \frac{2}{{{K}^{2}}}\sum\limits_{k=1}^{K}{\left( \mathbb{E}\left[ {{\left\| \eta _{t}^{k} \right\|}^{2}}|{{x}_{t}} \right] \right)}-\frac{2}{K}\mathbb{E}\left[ \left\| \bar{\eta}_t \right\|^2|x_t \right]+2\mathbb{E}\left[ \left\| \bar{\eta}_t \right\|^2|x_t \right]\\
            &=&\frac{2}{K}\mathbb{E}\left[ {{\left\| \eta _{t}^{1} \right\|}^{2}}|{{x}_{t}} \right]+2\left( 1-\frac{1}{K} \right)\mathbb{E}\left[ \left\| \bar{\eta}_t \right\|^2|x_t \right]\\
            &\le&
            \frac{{2C{L}^{2}}}{Kn}+\frac{2(1-{1}/{K})}{{{n}^{2}}}{{\left\| \nabla {{F}_{h_t}}({{x}_{t}}) \right\|}^{2}} \le \frac{{2C{L}^{2}}}{Kn}+\frac{{2{L}^{2}}}{{{n}^{2}}}\left( 1-\frac{1}{K} \right).
        \end{array}
    \]
    As a result we have the inequality:
    \begin{eqnarray}
        D({x}_t) &=& 2\rho_t\left(\mathbb{E} {{F}_{h_t}}({{x}_t})-{{F}_{h_t}}({{x}^*})\right)\nonumber \\
        &\le&
        n\mathbb{E}{{\left\| {{x}^{*}}-{{x}_{t}} \right\|}^{2}} - n\mathbb{E}{{\left\| {{x}^{*}}-{{x}_{t+1}} \right\|}^{2}} + \rho _{t}^{2}n\mathbb{E}\mathbb{E}\left[ \left. {{\left\| {{\eta }_{t}} \right\|}^{2}} \right|{{x}_{t}} \right]\nonumber \\
        &\le&
        n\mathbb{E}{{\left\| {{x}^{*}}-{{x}_{t}} \right\|}^{2}} - n\mathbb{E}{{\left\| {{x}^{*}}-{{x}_{t+1}} \right\|}^{2}} + \frac{2\rho _{t}^{2}L^2}{K}\left(C+\frac{K}{n}\right).\nonumber
    \end{eqnarray}
    Now let us sum these inequalities over $t$ from 1 to $T$,
    $$ 2\sum_{t=1}^T\rho_t\left(\mathbb{E}	{{F}_{h_t}}({{x}_t})-{{F}_{h_t}}({{x}^*})\right) \le n\mathbb{E}{{\left\| {{x}^{*}}-{{x}_{1}} \right\|}^{2}}+\frac{2L^2}{K}\left(C+\frac{K}{n}\right)\sum_{t=1}^T\rho _{t}^{2}. $$
    First, consider the case of the fixed smoothing parameter $h_t=h$ for all $t$ and let $x^*$ be such that $F_h(x^*)=\min_{x\in X}F_h(x)$. Let us subdivide both sides of the above inequality by $2\sum_{t=1}^T\rho_t$. For the average point $\bar{x}_T=\sum_{t=1}^T\rho_t x_t/\sum_{t=1}^T\rho_t $ due to convexity of $F_h$ we have
    \begin{eqnarray}
        \mathbb{E} {{F}_{h}}({\bar{x}_t})-\min_{x\in X}F_h(x)
        &\le&
        \frac{\sum_{t=1}^T\rho_t\left(\mathbb{E}	{{F}_{h}}({{x}_t})-\min_{x\in X}F_h(x)\right)}{\sum_{t=1}^T\rho_t}\nonumber \\
        &\le&
        \frac{n\mathbb{E}{{\left\| {{x}^{*}}-{{x}_{1}} \right\|}^{2}}}{2\sum_{t=1}^T\rho_t}+\frac{L^2}{K}\left(C+\frac{K}{n}\right)\frac{\sum_{t=1}^T\rho_t^2}{\sum_{t=1}^T\rho_t}
        \label{ineq13}
    \end{eqnarray}
    For the constant step size ${{\rho }_{t}}\equiv \rho $ and $\bar{x}_T=T^{-1}\sum_{t=1}^T\rho_t$, we obtain
    \begin{equation}
        \label{ineq14}
        \mathbb{E}{{F}_{h}}({{\bar{x}}_{T}})-\min_{x\in X}F_h(x)\le \frac{n{{D}^{2}}}{2T\rho }+\frac{{{L}^{2}}\rho }{K}\left( {C}+\frac{K}{n} \right).
    \end{equation}
    For minimizing the right hand side of (\ref{ineq14}), let us take $\rho =\frac{D\sqrt{nK}}{L\sqrt{2T(C+K/n)}}$, then we obtain
    \begin{eqnarray}
        \mathbb{E}{{F}_{h}}({{\bar{x}}_{T}})-\min_{x\in X}F_h(x)
        &\le&
        \frac{LD}{\sqrt{T}}\sqrt{\frac{2n}{K}}\left(C+\frac{K}{n}\right)^{1/2}.\nonumber
    \end{eqnarray}
    For the variable steps $\rho_t =\frac{D\sqrt{nK}}{L\sqrt{2t(C+K/n)}}$ and $\bar{x}_t=\sum_{\tau=1}^t \rho_t x_t/\sum_{\tau=1}^t \rho_t$, from (\ref{ineq13}) we obtain
    \begin{eqnarray}
        \mathbb{E}{{F}_{h}}({{\bar{x}}_{t}})-\min_{x\in X}F_h(x)
        &=& \frac{DL}{\sqrt{2}}\sqrt{\frac{n}{K}}\left(C+\frac{K}{n}\right)^{1/2} \frac{\left(1+\sum_{\tau=1}^t\tau^{-1}\right)}{\sum_{\tau=1}^t\tau^{-1/2}} \nonumber \\
        &\le& \frac{DL}{\sqrt{2}}\sqrt{\frac{n}{K}}\left(C+\frac{K}{n}\right)^{1/2} \frac{2+\ln{t}}{\sqrt{t}-1}. \nonumber
    \end{eqnarray}
    Now consider the case of a variable smoothing parameter $h_t$ and let $x^*$ be such that $F(x^*)=\min_{x\in X}F(x)$. Since $|F(x)-F_{h_t}(x)|\le Lh_t$ for $x\in X$, then $|F_{h_t}(x_t)-F(x_t)|\le Lh_t$, $|F_{h_t}(x^*)-F(x^*)|\le Lh_t$, and we obtain the estimate:
\[\begin{array}{l}
        2\sum_{t=1}^T\rho_t\left(\mathbb{E}	{F}({{x}_t})-\min_{x\in X}{{F}}({{x}})\right)
 =      2\sum_{t=1}^T\rho_t\left(\mathbb{E}	{F}({{x}_t})-{{F}}({{x}^*})\right) \\
 \;\;\;\;\;\;\;\;\;\;\;\;\;\;\;\;\;\;\;\;\;\;\;\;\;\;\;\;\;\;
\le    2\sum_{t=1}^T\rho_t\left(\mathbb{E}	{{F}_{h_t}}({{x}_t})-{{F}_{h_t}}({{x}^*})\right) 
         +4\sum_{t=1}^T L\rho_t h_t \\
 \;\;\;\;\;\;\;\;\;\;\;\;\;\;\;\;\;\;\;\;\;\;\;\;\;\;\;\;\;\;
\le   nD^2+\frac{2L^2}{K}\left(C+\frac{K}{n}\right)\sum_{t=1}^T\rho _{t}^{2}+4\sum_{t=1}^T L\rho_t h_t. 
\end{array}\]
    By dividing this inequality by $\sum_{t=1}^t\rho_\tau$ for $ \bar{x}_t=\sum_{\tau=1}^t\rho_\tau x_\tau/\sum_{\tau=1}^t\rho_\tau $ and $\rho_t =\frac{D\sqrt{nK}}{L\sqrt{2t(C+K/n)}}$, $h_t=L\rho_t/K$, we obtain:
\begin{eqnarray}
        &&\mathbb{E}{F}({\bar{x}_t})-\min_{x\in X}{{F}}({{x}})
        \le
        \frac{nD^2}{2\sum_{\tau=1}^t\rho_\tau} +\frac{L^2}{K}\left(C+\frac{K}{n}\right)\frac{\sum_{\tau=1}^t\rho _{\tau}^{2}}{\sum_{\tau=1}^t\rho_\tau} +2L\frac{\sum_{\tau=1}^t \rho_\tau h_\tau}{\sum_{\tau=1}^t\rho_\tau} \nonumber \\
        &&=
        \frac{nD^2}{2\sum_{\tau=1}^t\rho_\tau} +\frac{L^2}{K}\left(2+C+\frac{K}{n}\right)\frac{\sum_{\tau=1}^t\rho _{\tau}^{2}}{\sum_{\tau=1}^t\rho_\tau}
        =\frac{DL}{2}\sqrt{\frac{n}{K} \left(2+C+\frac{K}{n}\right)}\frac{2+\ln{t}}{\sqrt{t}-1}. \nonumber
    \end{eqnarray}
\end{proof}
\begin{theorem}
    \label{Theorem 3} (Gaussian smoothing with batch estimation of the finite-difference gradient). In conditions of Theorem \ref{Theorem 2}, where instead ball smoothing the gaussian smoothing is used, i.e. random vectors $\tilde{y}_{t}^{k}$   in (\ref{eqn15}) have the standard normal distribution, the following estimates hold true. At a constant step $\rho =\frac{D}{L\sqrt{2T(1+(n-1)/K)}}$ and a constant smoothing parameter $h_t=h$ the average point ${{\bar{x}}_{T}}=\frac{1}{T}\sum_{t=1}^{T}{{{x}_{t}}}$ satisfies the relation:
    \begin{equation}
        \mathbb{E}{{F}_{h}}({{\bar{x}}_{T}})-\min_{x\in X}{{F}_{h}}({{x}})
        \le 
        \frac{LD}{\sqrt{T}}\sqrt{2\left(1+\frac{n-1}{K}\right)}.\nonumber
    \end{equation}
    For steps $\rho_t =\frac{D}{L\sqrt{2t(1+(n-1)/K)}}$  and $\bar{x}_t=\sum_{\tau=1}^t \rho_t x_t/\sum_{\tau=1}^t \rho_t$ it holds
    \begin{equation}
        \mathbb{E}{{F}_{h}}({{\bar{x}}_{t}})-\min_{x\in X}{{F}_{h}}({{x}})
        \le
        {LD\sqrt{2}}\left({1+\frac{n-1}{K}}\right)^{1/2}\frac{2+\ln{t}}{\sqrt{t}-1}.
        \nonumber
    \end{equation}
    For the variable steps $\rho_t =\frac{D}{L\sqrt{t(1+(n+3)/K)}}$ and smoothing parameters $h_t=L\rho_t/K$ for the averaged points $\bar{x}_t=\sum_{\tau=1}^t \rho_t x_t/\sum_{\tau=1}^t \rho_t$ it holds
    \begin{eqnarray}
        \mathbb{E}{F}({\bar{x}_t})-\min_{x\in X}{{F}}({{x}})
        &\le&
        \frac{DL}{2}\left({1+\frac{n+3}{K}}\right)^{1/2}
        \frac{2+\ln{t}}{\sqrt{t}-1}.
    \end{eqnarray}
\end{theorem}
\begin{proof}
    The proof of this theorem basically repeats the proof of Theorem \ref{Theorem 2}, only instead of properties (\ref{eqn9}), (\ref{shamirinequality}) the relations (\ref{gaussfindif}), (\ref{Nesterov_Spokoiny}) are used. 
\end{proof}		
		There are related results on the rate of convergence of finite-difference methods for smooth functions (cf. \cite{Duchi_2011,Duchi_2015,shamir2017optimal}). In Theorems \ref{Theorem 2}, \ref{Theorem 3} we do not assume differentiability of the objective function. From the obtained  results (Theorem \ref{Theorem 2}) on the rate of convergence of the stochastic smoothing method on the class of convex nonsmooth functions, the following conclusions can be drawn regarding the parameters of the method in optimizing the smoothed function ${{F}_{h}}$. The step multiplier $\rho $ should be proportional to $\sqrt{n}$, the square root of the dimension of the space of the variables, and the estimate $D$ of the Euclidean distance from the starting point ${{x}_{0}}$ to the minimum point ${{x}^{*}}$.

The direction ${{\eta }_{t}}$ in the method should be normalized to the Lipschitz constant $L$ of the objective function, or to the estimation of this constant as in the adaptive methods Adam \cite{Kingma_Ba_2017} and AdaGrad \cite{Duchi_2011}. The batch size parameter $K$ allows to reduce dependence of the rate of convergence on the dimension $n$.

\subsection{Stochastic finite-difference gradient methods for nonconvex problems}

In the proposed approach to the optimization of nonsmooth and discontinuous functions under constraints, first the initial problem (\ref{eqn1}) is reduced by the method of exact nonsmooth penalties to the problem of unconditional optimization (\ref{eqn4}), (\ref{eqn5}) (see Sec. \ref{sec:penalty}), then the modified objective function is smoothed and approximated by smooth mathematical expectation functions (\ref{Steklov-Schwartz}), (\ref{ballsmoothing}), (\ref{gaussmoothing}) (Sec. \ref{sec:smoothing}).

In conditions of Theorem \ref{Theorem1} a sequence of smoothed averaged functions ${{F}_{h}}$  epi-graphically converges to $F$ as $h\to 0$. Then by \cite[Theorem 7.33]{RocW98} global minimums of ${{F}_{h}}$ converge to the global minimums of $F$ as $h\to 0$. Convergence of local minimums was studied in \cite{ENW95,ermol1998stochastic}. For the gradients of the smoothed functions stochastic finite-difference representations (\ref{eqn9}), (\ref{gaussfindif}) hold true. The latter can be used for optimization of the smoothed functions by contemporary efficient stochastic optimization methods \cite{Bottou_Curtis_Nocedal_2018}, in particular:
 Nemirovsky-Yudin trajectory averaging method \cite{Nemirovsky_Yudin_1983,Nemirovski_2009}, Polyak's method (heavy ball) \cite{Pol87};
Gelfand-Zetlin-Nesterov method (ravine step method) \cite{Gelfand_Tsetlin_1962,Nesterov_1983} and its stochastic generalizations;
Adaptive stochastic gradient methods \cite{Duchi_2011,Kingma_Ba_2017}.
Performance of the one of these method \cite{Nemirovski_2009} in the convex case is evaluated in Theorems \ref{Theorem 2}, \ref{Theorem 3}. Generalizations of these methods to the nonconvex smooth and nonsmooth cases are available in the works of \cite{Ermoliev_Norkin_2003,Gup79,Mai_Johansson_2020,MGN87,Ruszczynski_2020}.

\section{Global optimization by smoothing}

One heuristic background of applying smoothing method for the global optimization purposes is that smoothing eliminates shallow local minimums and little chainges wide and deep ones. The other justification is provided by properties of critical points/paths of smoothed functions, cf. \cite[p. 135-137]{MGN87}, \cite{Arikan_Burachik_Kaya_2020,Burachik_Kaya_2021,Norkin_2020}.

The smoothing method in global optimization consists in approximation of a multi-extremal function by a sequence of the so-called smoothed (or averaged) functions and optimization of the latter by some suitable available methods. When the smoothing parameter goes down, the averaged functions uniformly (or epi-graphically) converge to the original function, so the global minima of the approximate functions converge to the global minima of the original one.

The transition from a minimum point of one smoothed function to the starting point for minimizing another smoothed function is carried out by the Gelfand-Tsetlin-Nesterov ravine step method (\cite{Gelfand_Tsetlin_1962,Nesterov_1983}). In this method, if we have two points at the bottom of a valley, we can find the direction of the valley and guess a starting point to descend further to the bottom of the valley, and so on.

To minimize smoothed functions ${{F}_{h}}(x)$, any smooth optimization methods can be used, if their gradients can be calculated analytically, numerically, or estimated statistically. This is easy to do for one-dimensional functions \cite{Norkin_2020} and for separable functions $F(x)=\sum\nolimits_{i}{{{f}_{i1}}({{x}_{1}})\times ...\times {{f}_{in}}({{x}_{n}})}$. For example, \cite{Arikan_Burachik_Kaya_2020,Burachik_Kaya_2021} exploits analytical smoothing of the polynomial functions to derive differential equation for critical curves of the set $\{(x,h):\nabla{{F}_{h}}(x)=0, h\ge 0\}$.

In general, stochastic finite-difference methods (\cite{Bottou_Curtis_Nocedal_2018,Duchi_2011,Duchi_2015,Ermoliev_1976,Ermoliev_Norkin_2003,Gup79,Kingma_Ba_2017,MGN87,Pol87}) can be used for minimization of functions ${{F}_{h}}(\cdot)$ at each smoothing stage.

\section{Algorithm of the stochastic successive smoothing method}

Algorithm of the stochastic successive smoothing method consists of the following steps.
\begin{enumerate}
    \item Reduce the problem of conditional optimization to the problem of unconditional optimization of some coercive function;
    \item Select a decreasing sequence of smoothing parameters with a sufficiently large initial value of the smoothing parameter;
    \item Minimize successively the smoothed functions by some effective deterministic or stochastic method, using the results of minimizing the previous smoothed functions to start minimization of the subsequent smoothed function.
\end{enumerate}
The results of testing of the stochastic successive smoothing method on some small dimensional global optimization (box constrained) problems are presented in \cite{Knopov_Norkin_2022}, \cite{Norkin_2020,Norkin_2022}. Below we present an example of solution a larger test problem.
\begin{example}
    \label{Example 1} Largest Small Polygon \cite{Dolan_More_Munson_2004}. Find the polygon of maximal area, among polygons with $n$  sides and diameter $d\le 1$. If $({{r}_{i}},{{\theta }_{i}})$  are the coordinates of the vertices of the polygon, the problem is to find
    $$f_{n}^{*}=\underset{r,\,\varphi }{\mathop{\max }}\,\left[ f(r,\varphi )=\frac{1}{2}\sum\limits_{i=1}^{n-1}{{{r}_{i+1}}{{r}_{i}}\sin {{\varphi }_{i+1}}} \right]$$ 
    subject to the constraints:
\[\begin{array}{l}
        {{\left( r_{i}^{2}+r_{j}^{2}-2{{r}_{i}}{{r}_{j}} \cos \sum\nolimits_{k=i+1}^{j}{{{\varphi }_{k}}} \right)}^{{1}/{2}\;}}\le 1,\;\;\; 1\le i<j\le n; \\
        {{\theta }_{i}}=\sum\nolimits_{k=1}^{i}{{{\varphi }_{k}}},	1\le i\le n;	\;\;\;\sum_{i=1}^{n}{{{\varphi }_{i}}}\le \pi ; \\
        {{r}_{i}}\ge 0,\;\; 0\le {{\varphi }_{i}}\le {2\pi }/{n},\;\;	1\le i\le n;\;\;\;\;	{{r}_{1}}={{\varphi }_{1}}=0. 
\end{array}\]
    Here $r=({{r}_{1}},...,{{r}_{n}})$,  $\varphi =({{\varphi }_{1}},...,{{\varphi }_{n}})$. It is known that $f_{3}^{*}={\sqrt{3}}/{4}\;\approx 0.4330$, $f_{4}^{*}=0.5$, $\underset{n\to \infty }{\mathop{\lim }}\,f_{n}^{*}={\pi }/{4}\;\approx 0.7854$.  Optimal polygons approach to the circle of diameter 1 as $n\to \infty $. First we transform the constrained problem to an unconstrained problem $\underset{r,\,\varphi }{\mathop{\max }}\,{{F}_{3}}(r,\varphi )$  by the following penalty functions with penalty coefficients ${{p}_{1}},{{p}_{2}},{{p}_{3}}>0$: 	
    \begin{enumerate}
        \item ${{F}_{1}}(r,\varphi ) = \left\{ \begin{matrix} f(r,\varphi ), & \sum\nolimits_{i=1}^{n}{{{\varphi }_{i}}}\le \pi, \\ f(r,\lambda \varphi )+{{p}_{1}}\left( \sum\nolimits_{i=1}^{n}{{{\varphi }_{i}}}-\pi  \right), & otherwise, \\ \end{matrix} \right.$, $\lambda ={\pi }/{\sum\nolimits_{i=1}^{n}{{{\varphi }_{i}}}}\;$,  ${{p}_{1}}=1$;    
        \item ${{F}_{2}}(r,\varphi )={{F}_{1}}(r,\varphi )+{{p}_{2}}\max \left\{ 0,\sum\nolimits_{i<j}{\left( \sqrt{r_{diag}} - 1 \right)} \right\}$,  ${{p}_{2}}=1$, $ r_{diag} = r_{i}^{2}+r_{j}^{2}-2{{r}_{i}}{{r}_{j}}\cos \sum\nolimits_{k=i+1}^{j}{{{\varphi }_{k}}} $;
        \item ${{F}_{3}}(r,\varphi )={{F}_{2}}(\hat{r},\hat{\varphi })+{{p}_{3}}\left( \left\| r-\hat{r} \right\|+\left\| \varphi -\hat{\varphi } \right\| \right)$,	${{p}_{3}}=10$.	
    \end{enumerate}
    Here $\hat{r}$ and $\hat{\varphi }$  are projections of vectors $r$ and $\varphi $ on the cubes ${{[0,1]}^{n}}$ and ${{[0,{2\pi }/{n}\;]}^{n}}$, respectively. We solved the problem  $\underset{r,\,\varphi }{\mathop{\max }}\,{{F}_{3}}(r,\varphi )$ by the stochastic smoothing method, results are presented in the Table 1. Further examples of solving stochastic storage optimal control problems by the stochastic successive smoothing method can be found in \cite{Knopov_Norkin_2022}.
\end{example}
\begin{table*}[h] \label{polygon}
    \caption{Optimal polygons}
    \begin{tabular}{rrrrrrrrr}
        \hline
            $n$ 	&3	&4		&20	&50	&100	&200	&500\\
            Smoothing \\Iterations	&100	&200		&500	&1000	&2000	&2000	&4000\\
            Func. calc.	&4040	&11256		&132264	&620620	&2465232	&3521760	&15627906\\
            Max. achived	&0.4300	&0.4994		&0.7680	&0.7763	&0.7788	&0.7697	&0.7588\\
            Ideal value 	&0.4330	&0.5000		&0.7854	&0.7854	&0.7854	&0.7854	&0.7854\\
        \hline
    \end{tabular}
\end{table*}

\section{Conclusions}	
In the paper, the stochastic successive smoothing method (\cite{Norkin_2020}) is extended to general conditional optimization problems with nonsmooth and discontinuous functions. Initially, the conditional problem is transformed to an unconditional one by a modified exact penalty function method, in which one need not to select the value of the penalty factor. Then the unconditional problem is approximated by a sequence of smoothed (averaged) functions starting from strong smoothing with a gradual decrease in the degree of smoothing. Modern methods of stochastic gradient optimization can be used to optimize approximate smoothed functions. In this respect, we show that gradients of the smoothed functions can be estimated by means of finite differences in stochastic directions. The transition from the optimization of one smoothed function to the optimization of the next function is carried out using the Gelfand-Zetlin-Nesterov method but unlike this method for the local decent we use many step stochastic procedure. The applied sequential smoothing eliminates shallow local minima of the original problem, so the method gets some global properties. 

\section{Acknowledgements}
The research of the first author was supported by the Volkswagen Foundation (Az.: 9C090, 2022) and the National Research Fund of Ukraine (Project 2020.02/0121). The research of the second author was supported by the German Research Foundation (Project 416228727–SFB 1410). The work of the third author was supported by the National Research Fund of Ukraine (Project 2020.02/0121).

\bibliographystyle{splncs04}
\bibliography{paper}

\end{document}